\documentclass{amsart}

\usepackage{lineno}
\usepackage[all]{xy}
\usepackage{amsmath}
\allowdisplaybreaks
\usepackage{amssymb}
\usepackage{array}
\usepackage{color}
\usepackage{verbatim}

\theoremstyle{plain}
\newtheorem{theorem}{Theorem}[section]
\newtheorem{lemma}[theorem]{Lemma}

\newtheorem{corollary}[theorem]{Corollary}

\theoremstyle{definition}
\newtheorem{definition}[theorem]{Definition}

\theoremstyle{remark}
\newtheorem{remark}[theorem]{Remark}

\def\Q{{\mathbb Q}}

\def\cat0{\mathrm{cat}_0}

\def\aut{\mathrm{aut}}

\def\ad{\mathrm{ad}}

\numberwithin{equation}{section}

\begin{document}
\author{Yang Bai}
\address{School of Mathematical Sciences, Nankai University, Tianjin 300071, P.R.China}
\email{1687694204@qq.com}

\author{Xiugui Liu}
\address{School of Mathematical Sciences and LPMC, Nankai University, Tianjin 300071, P.R.China}
\email{xgliu@nankai.edu.cn}

\author{Chaoshun Yang}
\address{School of Mathematical Sciences, Nankai University, Tianjin 300071, P.R.China}
\email{1296736185@qq.com}

\thanks{The second author was supported in part by the National Natural Science Foundation of China (Grant No. 12171165).}
\date{}
\title{Realization of permutation modules as homology of CW-complexes}

\begin{abstract}
In this paper, we investigate the realizability problem, which asks whether prescribed group actions on graded modules can be realized by the groups of self-homotopy equivalences of CW-complexes acting on their homology. We provide a partial answer in the case of arbitrary groups and permutation modules concentrated in certain degrees. As a consequence, we realize every group as the group of self-homotopy equivalences of an $R$-local CW-complex with arbitrary prescribed connectivity, where $R$ may be chosen with $\rho(R)$ sufficiently large. In particular, this provides a complete answer to Kahn's realizability problem. 

\end{abstract}

\maketitle

\section{Introduction}

Realizability problems are classical in algebraic topology, especially those concerning homotopy groups, homology groups, and groups of self-homotopy equivalences. In \cite{C}, Costoya, Gomes and Viruel raised the following realizability problem of group actions, which can be considered as the dual of \cite[Question 1.3]{CV1}.

\noindent {\bf Question 1.1.} For a given group $G$, decide for which $\mathbb{Z}G$-modules $M$ there exists a topological space $X$ such that
\begin{enumerate}
    \item $H_{\ast }(X;\mathbb{Z})\cong M$;
    \item the group of self-homotopy equivalences $\mathcal{E}(X)$ is isomorphic to $G$;
    \item the action of $G$ on $M$ is equivalent to the canonical action of $\mathcal{E}(X)$ on $H_{\ast }(X;\mathbb{Z})$.
\end{enumerate}

They proved the existence of such $X$ as finite spaces when $G$ is a finite group and $M$ is a finitely generated generalized permutation $\mathbb{Z}G$-module \cite[Theorem A]{C}. Their approach was inspired by Chocano, Mor\'on and Ruiz~del~Portal, who effectively employed Alexandroff spaces in several realizability problems \cite{Al,A}, and relied on the following realization result of permutation representations as graphs.

\begin{theorem}\cite[Theorem 3.1]{CC}\label{qzy}
    Let $G$ be an arbitrary group, let $V$ be a set and let $\rho: G\rightarrow \mathrm{Sym}(V)$ be a (not necessarily faithful) permutation representation. Then, there is a simple, undirected graph $\mathcal{G}$ such that

    \begin{enumerate}
        \item $V\subset V(\mathcal{G})$ and $V$ is invariant under the $\mathrm{aut}(\mathcal{G})$-action;
        \item $\mathrm{aut}(\mathcal{G}) \cong G$;
        \item the restriction map $G\cong \mathrm{aut}(\mathcal{G}) \rightarrow \mathrm{Sym}(V)$ is precisely $\rho$;
        \item there is a faithful representation $\bar{\rho}: G\cong \mathrm{aut}(\mathcal{G}) \rightarrow  \mathrm{Sym}( V(\mathcal{G}) \backslash V)$ such that the restriction map $G\cong \mathrm{aut}(\mathcal{G}) \rightarrow \mathrm{Sym}(V(\mathcal{G}))$ is $\rho \oplus \bar{\rho}$.
    \end{enumerate}
\end{theorem}

In addition to the results mentioned above, relevant realizability problems have also been investigated within the category of simply-connected CW-complexes \cite{MB,CV0,CV1,CV,DK1}. However, such $X$ in Question 1.1 may fail to exist in this category even when $G$ is a finite group and $M$ is a finitely generated generalized permutation $\mathbb{Z}G$-module. In particular, a simple counterexample is given by taking $M=H_{\ast}(S^2)$, for which $\mathcal{E}(X)\cong \mathbb{Z}_2$ is uniquely determined. Therefore, we focus on a weaker version of Question 1.1 in the category of simply-connected CW-complexes: rather than realizing $M$ as the homology of $X$ in all degrees, we realize it only in finitely many consecutive degrees.

Let $R\subset \mathbb{Q}$ be a subring, and let $\rho(R)$ denote the least prime which is not invertible in $R$ (if none exists put $\rho(R)=\infty$). Recall that a {\it permutation $RG$-module} is a free $R$-module on a $G$-set whose $RG$-module structure is defined by $R$-linearly extending the $G$-action on the $G$-set. We provide a partial answer in the particular case where $G$ is an arbitrary group and $M$ is a permutation $RG$-module concentrated in finitely many consecutive degrees. 

\begin{theorem}\label{fi}
    Let $k\geq 2$ be an integer and $G$ an arbitrary group. Let $\left\{M_i\right\}_{i=1}^{k}$ be a finite collection of permutation $RG$-modules. Suppose that there exists an integer $n$ satisfying
    $$n\geq 3k+5, \ \  \rho(R) >20n^3-33n^2+8n+4. $$
    Then, there exists an $(8n^2-6n+1)$-connected finite-dimensional CW-complex $X$ such that
    \begin{enumerate}
        \item $\mathcal{E}(X_{R})\cong G$;
        \item $H_{k(8n^2-2n-2)+i+1}(X_{R};\mathbb{Z})\cong M_i$ for $i=1,\dots, k$;
        \item the action of $\mathcal{E}(X_{R})$ on $H_{k(8n^2-2n-2)+i+1}(X_{R};\mathbb{Z})$ is equivalent to the action of $G$ on $M_i$ for $i=1,\dots, k$.
    \end{enumerate}
    Here $X_{R}$ denotes the $R$-localization of $X$. In particular, $X$ can be chosen to be a finite CW-complex when $G$ is a finite group and $\left\{M_i\right\}_{i=1}^{k}$ is a collection of finitely generated $RG$-modules.
\end{theorem}

\begin{remark}
    Theorem \ref{fi} deals with the case $k\geq 2$. The case $k=1$ can be settled by applying Theorem \ref{fi} to the arbitrary group $G$ and the collection of permutation $RG$-modules $\left\{M_i\right\}_{i=1}^{2}=\{M_1,R\}$ with the trivial $RG$-module $R$.
\end{remark}

The realizability problem proposed by Kahn \cite{DK1}, which seeks to characterize all groups that appear as the group of self-homotopy equivalences of simply-connected CW-complexes, has a long history. Costoya and Viruel solved the case of finite groups \cite{CV0}, and generalized their earlier results to a broad class of orthogonal groups \cite{CV1}.
Subsequently, Benkhalifa \cite{MB} claimed a complete answer to Kahn's realizability problem for arbitrary groups, with a revised version \cite{MB1} on arXiv. However, after a close examination we reveal some minor miscalculations in \cite{MB,MB1}: the arguments depend on the non-exactness of certain Lie monomials, whereas some of them are in fact boundaries. For instance, using the notation from \cite{MB,MB1}, the bracket
$$Z=\Big[(\ad w_1)^5(w_3),\big[(\ad w_2)^3(w_3),(\ad w_2)^2([w_3,w_5])\big]\Big]$$
is claimed to be non-exact (\cite[Lemma 3.3]{MB} and \cite[Lemma 3.5]{MB1}), but a direct computation shows that it is a boundary: 
$$Z=\partial\Big(\frac{1}{2}\Big[(\ad w_1)^5(w_3),\big[(\ad w_2)^3(w_3),(\ad w_4)([w_3,w_5])\big]\Big]\Big).$$

The construction in this paper is motivated by the general framework of \cite{MB,MB1}, but we introduce essential modifications to avoid the issues discussed above. As an immediate consequence of Theorem \ref{fi}, the following corollary indeed provides a corrected and complete answer to Kahn's realizability problem in the category of CW-complexes with arbitrary prescribed connectivity. 

\begin{corollary}\label{co1}
    For every group $G$ and positive integer $N$, there exists a positive integer $P_{G,N}$ such that, for every subring $R\subset \Q$ with $\rho(R)>P_{G,N}$, there exists an $N$-connected finite-dimensional CW-complex $X$ satisfying $\mathcal{E}(X_{R})\cong G$, where $X_{R}$ is the $R$-localization of $X$. In particular, $X$ can be chosen to be a finite CW-complex when $G$ is a finite group.
\end{corollary}

The paper is organized as follows. We start by recalling basic facts about graphs and Anick's $R$-local homotopy theory in Section 2. Then in Section 3, we provide the construction of a differential graded Lie $R$-algebra with a given graph. In Section 4, we investigate the self-homotopy equivalences of the constructed differential graded Lie $R$-algebra with a series of lemmas, and finally prove Theorem \ref{fi} and Corollary \ref{co1}.

\section{Preliminaries}

For basic facts about graphs, we refer to \cite{GT}. Only simple undirected graph $\mathcal{G}=(V(\mathcal{G}),E(\mathcal{G}))$ are considered. This means they do not have loops and they are not directed, that is, for any vertex $v$ in $V(\mathcal{G})$, the edge $(v,v)$ is not in $E(\mathcal{G})$, if an edge $(v,w)$ is in $E(\mathcal{G})$, then $(w,v)$ is also in $E(\mathcal{G})$. The automorphism group of a graph $\mathcal{G}$ is denoted by $\aut (\mathcal{G})$.

We refer to \cite{AN1,AN2} for basic facts about Anick's $R$-local homotopy theory. Let $R\subset \mathbb{Q}$ be a subring, and let $\rho(R)$ denote the least prime which is not invertible in $R$ (if none exists put $\rho(R)=\infty$).

A {\it graded Lie $R$-algebra} is a graded $R$-module $L=\oplus_{i=-\infty}^{\infty} L_i$, together with a bilinear pairing $[\ ,\ ]: L_i\otimes L_j\rightarrow  L_{i+j}$ satisfying the graded skew-symmetry and the graded Jacobi identity. The degree of a homogeneous element $x$ is denoted by $|x|$, and the derivation $\ad x$ is defined by $(\ad x)(y)=[x,y],\ \forall y\in L$. If $L_{i}=0$ for $i<r$, we call $L$ {\it $r$-reduced}.

The free $R$-module on the graded set $S$ is denoted by $\langle S \rangle$, and $S$ is called its {\it free basis}. If a graded Lie $R$-algebra $L$ is free on the free $R$-module $W$, we write $L=\mathbb{L}(W)$. Note that the universal enveloping algebra of $\mathbb{L}(W)$ is exactly the tensor algebra $TW$.

A {\it differential graded Lie $R$-algebra} (dgl for short) is a $1$-reduced graded Lie $R$-algebra $L$ together with a derivation $d$ of degree $-1$ whose square is zero. Denote by $DGL_r^m(R)$ the category of $r$-reduced differential graded Lie $R$-algebras which are free as Lie $R$-algebras and generated in degree at most $m$.
\begin{definition}[\cite{AN1}, p. 426]
    A free dgl over $R$ is called {\it $r$-mild} if it is generated as a Lie algebra by the range of dimensions $r$ through $r\rho(R)-1$.
\end{definition}

If $(\mathbb{L}(W),d)$ is an $r$-mild dgl, its {\it Tanr\'e cylinder} $I(\mathbb{L}(W),d)$ is constructed as follows \cite{AN1}:
$$I(\mathbb{L}(W),d)=(\mathbb{L}(W\oplus sW\oplus DsW),D),\ \ (sW)_i\cong W_{i-1},\ \ W\cong DsW,$$
where the differential $D$ is given by
$$ D(x)=dx, \ \ D(sx)=Dsx, \ \ D(Dsx)=0, \ \ \mathrm{for }\ x\in W.$$
Define a derivation $S$ of degree $+1$ on $(\mathbb{L}(W\oplus sW\oplus DsW),D)$ by setting
$$S(x)=sx,\ \ S(sx)=S(Dsx)=0, \ \ \mathrm{for }\ x\in W.$$
A homotopy between two dgl morphisms $\varphi_1,\varphi_2: (\mathbb{L}(W),d) \rightarrow (L,d^\prime)$ is a dgl morphism
$$F:(\mathbb{L}(W\oplus sW\oplus DsW),D) \rightarrow (L,d^\prime)$$
such that $F(x)=\varphi_1(x)$ and $F\circ e^{[D,S]}(x)=\varphi_2(x)$ for $x\in W$. The group of self-homotopy equivalences of an $r$-mild dgl $(\mathbb{L}(W),d)$ is denoted by $\mathcal{E}(\mathbb{L}(W))$.

The homotopy theory of $r$-mild dgl's is well-behaved. Our proof relies on the homotopy extension property \cite[p. 426]{AN1}.
\begin{theorem}[Homotopy extension property]\label{HEP}
    Let $(\mathbb{L}(W),d)\rightarrow (\mathbb{L}(W\oplus U),d)$ be an inclusion of $r$-mild dgl's. Let $f:(\mathbb{L}(W\oplus U),d)\rightarrow (L,d^\prime)$ be a morphism between $r$-mild dgl's, and let $F:I\mathbb{L}(W)\rightarrow L$ be a homotopy between $f|_{\mathbb{L}(W)}$ and some other morphism $g$. Then there is an extension $\widetilde{F}:I\mathbb{L}(W\oplus U)\rightarrow L$ of $F$ to a homotopy between $f$ and an extension $\widetilde{g}$ of $g$.
\end{theorem}

The following lemma treats a simple but important case in which two dgl morphisms are homotopic (cf. \cite[Lemma 3.4]{AN1} and \cite[Lemma 2.1]{MB}).
\begin{lemma}\label{ho}
    Let $(\mathbb{L}(W_{\leq k}),d)$ and $(\mathbb{L}(W^\prime),d^\prime)$ be $r$-mild dgl's. Let
    $$\varphi_1,\varphi_2: (\mathbb{L}(W_{\leq k}),d) \rightarrow (\mathbb{L}(W^\prime),d^\prime)$$
    be two morphisms satisfying
    \begin{align*}
    \varphi_2(x)&=\varphi_1(x), \ \ x\in W_{<k},  \\
    \varphi_2(z)&=\varphi_1(z)+ d^\prime m_z, \ \ z\in W_{k},
    \end{align*}
    for some $m_z\in \mathbb{L}(W^\prime)$. Then $\varphi_1$ and $\varphi_2$ are homotopic.
\end{lemma}

Let $CW_r^{m+1} $ denote the category of $r$-connected CW-complexes of dimension $\leq m+1$ having trivial $r$-skeleton, with arbitrary pointed maps among them. Define the homotopy category for $CW_r^{m+1} $ by collapsing $R$-local homotopy classes of maps, and likewise for $DGL_r^m (R)$. Under the condition $m<\min(r+2\rho(R)-3, r\rho(R)-1)$, Anick established an equivalence between the homotopy categories of $CW_r^{m+1}$ and $DGL_r^m (R)$ (see \cite[Section 8]{AN1} and \cite[Theorem 1]{AN2}).

\section{From graphs to Lie algebras}

Fix the integers $k,n$ with
$$k\geq2,\ \ n\geq 3k+5,$$
and a subring $R\subset \mathbb{Q}$ satisfying
$$\rho(R) > 20n^3-33n^2+8n+4.$$
According to the definition, for $i=1,2,\dots,k$, the permutation $RG$-module $M_i$ is a free $R$-module on a $G$-set, denoted by $V_i$. Applying Theorem \ref{qzy} to $(\cup_{i=1}^{k} V_i )\cup \{\ast\}$, there is a simple, undirected graph $\mathcal{G}=(V(\mathcal{G}),E(\mathcal{G}))$ with $V=(\cup_{i=1}^{k} V_i) \cup \{\ast\} \subset V(\mathcal{G})$ satisfying conditions (1)--(4). Here the extra element $\{\ast\}$ is fixed by $G$, and ensures that the $G$-set $V$ in Theorem \ref{qzy} is non-empty when every $M_i$ is zero. According to the construction in \cite[Section 3]{CC}, we can further assume that $\mathcal{G}$ is a connected graph with more than one vertex. Set $V_0 =V(\mathcal{G}) \backslash (\cup_{i=1}^{k} V_i)$. For convenience, write $h_v=i$ if $v\in V_i\subset V(\mathcal{G})$ for $i=0,1,\dots,k$. Each $V_i$ is invariant under $\aut (\mathcal{G})$ because the restricted action is the prescribed action on $V$; hence $h_{\sigma(v)}=h_v$ for every $\sigma\in \aut (\mathcal{G})$. 

We define a differential graded Lie $R$-algebra $\mathcal{L}^{\mathcal{G},0}$ as
\begin{equation*}
    \mathcal{L}^{\mathcal{G},0}=\left(\mathbb{L}(x_1,x_2,y,z,t,a_v,b,b_{(v,w)} | \ v\in V(\mathcal{G}), (v,w)\in E(\mathcal{G})),d\right),
\end{equation*}
where the degrees and differentials are described by
\begin{align*}
    |x_1|=|x_2|&=8n^2-6n+1, \\
    |y|&=8n^2-2n-2, \\
    |z|&=|y|+1=8n^2-2n-1, \\
    |t|&=2|x_2|+1=16n^2-12n+3, \\
    |a_{v}|&=(k-h_v)|y|+h_v|z|=k(8n^2-2n-2)+h_v, \\
    |b|=|b_{(v,w)}|&=(5n-7)|z|+|y|+1=40n^3-58n^2+7n+6,
\end{align*}
and
\begin{align*}
    &dx_1=dx_2=dy=dz=da_{v}=0,  \ \ \ \ \ dt=[x_2,x_2],  \\[6pt]
    &db=(\ad z)^{5n-7}(y)+(\ad z)^{n-6}(\ad x_1)^{4n+1}(y)+(\ad z)^{n-6}(\ad x_1)^{4n}([x_2,y]) \\
    &\ \ \ +(\ad y)^{2n}(\ad z)^{n-5}(\ad x_1)^{2n-1} (x_2) +(\ad y)^{2n+1}(\ad z)^{n-6}(\ad x_1)^{2n-3}([x_2,t]), \\[6pt]
    &db_{(v,w)}=(\ad a_{v})^3(\ad x_1)^{2n-1} (\ad y)^{2n-3(k-h_v)}(\ad z)^{n-5-3h_v}(x_2)  \\
    &\ \ \ \ \ \ \ \ \ +(\ad a_{v})(\ad a_{w})(\ad x_1)^{2n-3}(\ad y)^{2n+1-(2k-h_v-h_w)}(\ad z)^{n-6-(h_v+h_w)}([x_2,t]) .
\end{align*}

\smallskip

The differential $d$ is well-defined by direct computation, and the following is an important observation.
\begin{lemma}\label{lm1}
    The graded $R$-module $\mathcal{L}^{\mathcal{G},0}$ is trivial in degree $|b|+3n$.
\end{lemma}
\begin{proof}
    For $n\geq 3k+5\geq 11$, we have
    $$3n<|x_1|=|x_2|<|y|<|z|<|t|<|a_{v}|<|b|=|b_{(v,w)}|.$$
    It suffices to show the following equation
    \begin{equation}\label{eq0}
        k_1|x_1|+k_2|y|+k_3|z|+k_4|t|=|b|+3n
    \end{equation}
    with variables $k_1,k_2,k_3,k_4$ admits no non-negative integer solution.

    Suppose the equation has a non-negative integer solution $(k_1,k_2,k_3,k_4)$. By direct computation, we have
    $$k_1\leq 5n-4, \ \ k_2\leq 5n-6, \ \ k_3\leq 5n-6, \ \ k_4\leq \frac{5n}{2}-2.$$
    Note that if $(k_1,k_2,k_3,k_4)$ is a solution, then $(k_1+2,k_2-1,k_3+1,k_4-1)$ is also a solution. By induction, we can reach a solution with the second or fourth entry being zero. Hence, we may suppose $k_2$ or $k_4$ equals zero without loss of generality.

    {\bf Case 1}: $k_2=0$.

    After substituting $k_2$ and the degrees of the generators, Equation {(\ref{eq0})} becomes
    $$(2n-1)(4n-1)k_1+(2n-1)(4n+1)k_3+(2(2n-1)(4n-1)+1)k_4=40n^3-58n^2+10n+6.$$
    Modulo both sides by $2n-1$, we obtain the congruence
    $$k_4\equiv n+1 \ \ (\mathrm{mod} \ 2n-1).$$
    Since $k_4\leq \frac{5n}{2}-2$, it follows that $k_4=n+1$. Then we have
    \begin{equation}\label{eq1}
        (4n-1)k_1+(4n+1)k_3=12n^2-25n-3.
    \end{equation}
    Modulo both sides by $4n-1$, we get
    $$2k_3\equiv 2n-9 \ \ (\mathrm{mod} \ 4n-1).$$
    Since $k_3\leq 5n-6$ and $n\geq 11$, it follows that $k_3=3n-5$. By Equation (\ref{eq1}), we have $k_1=-2$, a contradiction.

    {\bf Case 2}: $k_4=0$.

    After substituting $k_4$ and the degrees of the generators, Equation (\ref{eq0}) becomes
    $$(2n-1)(4n-1)k_1+((2n-1)(4n+1)-1)k_2+(2n-1)(4n+1)k_3=40n^3-58n^2+10n+6.$$
    Modulo both sides by $2n-1$, we obtain the congruence
    $$-k_2\equiv n+1 \ \ (\mathrm{mod} \ 2n-1).$$
    Since $k_2\leq 5n-6$ and $n\geq 11$, it follows that $k_2=n-2$ or $3n-3$.

    {\bf Subcase 2.1}: $k_2=n-2$.

    After substitution, Equation (\ref{eq0}) becomes
    \begin{equation}\label{eq2}
        (4n-1)k_1+(4n+1)k_3=16n^2-12n-2.
    \end{equation}
    Modulo both sides by $4n-1$, we get
    $$2k_3\equiv -4 \ \ (\mathrm{mod} \  4n-1).$$
    Since $k_3\leq 5n-6$ and $n\geq 11$, it follows that $k_3=4n-3$. By Equation (\ref{eq2}), we have $k_1=-1$, a contradiction.

    {\bf Subcase 2.2}: $k_2=3n-3$.

    After substitution, Equation (\ref{eq0}) becomes
    \begin{equation}\label{eq3}
        (4n-1)k_1+(4n+1)k_3=8n^2-10n.
    \end{equation}
    Modulo both sides by $4n-1$, we get
    $$2k_3\equiv -2 \ \ (\mathrm{mod} \  4n-1).$$
    Since $k_3\leq 5n-6$ and $n\geq 11$, it follows that $k_3=4n-2$. By Equation (\ref{eq3}), we have $k_1=-2(n+1)$, a contradiction.

    This completes the proof of Lemma \ref{lm1}.

\end{proof}

Now we shall construct $\mathcal{L}^{\mathcal{G}}$ from $\mathcal{L}^{\mathcal{G},0}$ inductively. For $s=1,2,\dots ,3n$, denote by  $\Phi_{|b|+s}$ the suspension of the graded $R$-module $Z_{|b|+s-1}(\mathcal{L}^{\mathcal{G},s-1})$, where $Z_{|b|+s-1}(\mathcal{L} ^{\mathcal{G},s-1})$ denotes the set of cycles having degree $|b|+s-1$ in $\mathcal{L}^{\mathcal{G},s-1}$. Note that every subring of the rationals is a PID. Thus $\Phi_{|b|+s}$ is a free $R$-module. Therefore, we can inductively define a differential graded Lie $R$-algebra
$$\mathcal{L}^{\mathcal{G},s}=\mathcal{L}^{\mathcal{G},s-1} \sqcup   \mathbb{L}( \Phi_{|b|+s} )$$
by extending the differential $d$ as the following composition
$$\Phi_{|b|+s}\xrightarrow[\cong]{s^{-1}} Z_{|b|+s-1}(\mathcal{L}^{\mathcal{G},s-1} )\hookrightarrow \mathcal{L}_{|b|+s-1}^{\mathcal{G},s-1}$$
for $s=1, 2, \dots ,3n$. Finally, we obtain a dgl $\mathcal{L}^{\mathcal{G},3n}$ that is free as a Lie $R$-algebra. Set $$\mathcal{L}^{\mathcal{G}}=\mathcal{L}^{\mathcal{G},3n}.$$

\begin{remark}
    In our construction, $\mathcal{L}^{\mathcal{G},0}$ is a minimal differential graded Lie $R$-algebra, while $\mathcal{L}^{\mathcal{G}}$ may not be minimal.
\end{remark}

The following lemma is a straightforward result.
\begin{lemma}\label{lm2}
    For $\mathcal{L}^{\mathcal{G}}$ constructed as above, we have:
    \begin{enumerate}
        \item $\mathcal{L}^{\mathcal{G}}$ is $(8n^2-6n+1)$-mild;
        \item $\mathcal{L}^{\mathcal{G}}_{|x_1|}=\langle x_1,x_2 \rangle, \ \mathcal{L}^{\mathcal{G}}_{|y|}=\langle y \rangle, \ \mathcal{L}^{\mathcal{G}}_{|z|}=\langle z \rangle, \ \mathcal{L}^{\mathcal{G}}_{|t|}=\langle t \rangle$;
        \item For $i=0,1,\dots, k$, $\mathcal{L}^{\mathcal{G}}_{(k-i)|y|+i|z|}=\langle a_{v}  |\   h_v=i \  \rangle \oplus \mathbb{L}^{k-i,i}(y,z)$, where $\mathbb{L}^{k-i,i}(y,z)$ denotes the submodule of $\mathbb{L}(y,z)$ generated by Lie monomials with wordlength $k-i$ for $y$ and wordlength $i$ for $z$.
        \item For $s=0,1,\dots,3n$, $H_{|b|+s}(\mathcal{L}^{\mathcal{G}})=0$.
    \end{enumerate}
\end{lemma}
\begin{proof}
    (1) Put 
    $$r=8n^2-6n+1,\ \ m=40n^3-58n^2+10n+6. $$
    Recall that $n\geq 3k+5 \geq 11$ and $  \rho(R) >20n^3-33n^2+8n+4$. Since $\mathcal{L}^{\mathcal{G}}$ is generated as a Lie $R$-algebra by a range of dimension through $r$ to $m$, we have
    $$r \rho(R)-1> r+2\rho(R) -3>40n^3-58n^2+10n+6=m$$
    by direct computation. By definition, $\mathcal{L}^{\mathcal{G}}$ is $(8n^2-6n+1)$-mild.

    (2) It directly follows from the following inequalities for $n\geq 11$
    \begin{equation*}
        \begin{split}
    &|x_1|=|x_2|<|y|<|z|<2|x_1|, \\
    &2|x_1|<|t|=2|x_1|+1<|x_1|+|y| .
        \end{split}
    \end{equation*}

    (3) Fix $i$. Recall that $n\geq 3k+5$ and $|a_{v}|=(k-h_v)|y|+h_v|z|$ for $v\in V(\mathcal{G})$. Since $|x_1|>k\geq i$, it suffices to show the following equation
    \begin{equation*}
        k_1|x_1|+k_2|y|+k_3|z|+k_4|t|=(k-i)|y|+i|z|
    \end{equation*}
    with variables $k_1,k_2,k_3,k_4$ has a unique non-negative integer solution $(k_1,k_2,k_3,k_4)=(0,k-i,i,0)$, which can be derived by similar arguments as in Lemma \ref{lm1}.

    (4) Since $|x_1|>3n$ for $n\geq 11$, we have
    $$\mathcal{L}_{|b|+s}^{\mathcal{G}}=\mathcal{L}_{|b|+s}^{\mathcal{G},0}\oplus \Phi_{|b|+s} =\mathcal{L}_{|b|+s}^{\mathcal{G},s}$$
    for $s=1,2,\dots,3n$. Thus, we have
    $$B_{|b|+s}(\mathcal{L}^{\mathcal{G}})=B_{|b|+s}(\mathcal{L}^{\mathcal{G},0})+d\Phi_{|b|+s+1}=Z_{|b|+s}(\mathcal{L}^{\mathcal{G},s})$$
    for $s=0,1,\dots, 3n-1$, where $B_{|b|+s}(\mathcal{L}^{\mathcal{G}})$ and $B_{|b|+s}(\mathcal{L}^{\mathcal{G},0})$ denote the boundaries of degree $|b|+s$ in $\mathcal{L}^{\mathcal{G}}$ and $\mathcal{L}^{\mathcal{G},0}$, respectively. It follows that
    $$H_{|b|+s}(\mathcal{L}^{\mathcal{G}})=\frac{Z_{|b|+s}(\mathcal{L}^{\mathcal{G}})}{B_{|b|+s}(\mathcal{L}^{\mathcal{G}})}=\frac{Z_{|b|+s}(\mathcal{L}^{\mathcal{G}})}{Z_{|b|+s}(\mathcal{L}^{\mathcal{G},s})}=0$$
    for $s=0,1,\dots, 3n-1$. Finally, Lemma \ref{lm1} gives $\mathcal{L}^{\mathcal{G},0}_{|b|+3n}=0$. Then according to the construction of $\Phi_{|b|+3n}$, there are no cycles in degree $|b|+3n$. Hence we have 
    $$H_{|b|+3n}(\mathcal{L}^{\mathcal{G}})=\frac{Z_{|b|+3n}(\mathcal{L}^{\mathcal{G}})}{B_{|b|+3n}(\mathcal{L}^{\mathcal{G}})}=Z_{|b|+3n}(\mathcal{L}^{\mathcal{G}})=0.$$
    This completes the proof of Lemma \ref{lm2}.

\end{proof}

\section{Characterization of self-homotopy equivalences}

Since $\Phi_{|b|+s}$ is a free $R$-module for $s=1,2,\dots,3n$, we may choose a free basis $\{\phi_{s,\alpha}\}_{\alpha\in I_s}$ for $\Phi_{|b|+s}$. By virtue of Lemma \ref{lm2}\,(1), we proceed to characterize the group $\mathcal{E}(\mathcal{L}^{\mathcal{G}})$ of self-homotopy equivalences.
\begin{lemma}\label{hx1}
    Every $\sigma\in \mathrm{aut}(\mathcal{G})$ induces an automorphism $f_{\sigma}$ of $\mathcal{L}^{\mathcal{G}}$.
\end{lemma}
\begin{proof}
    We first define a morphism $g_{\sigma}$ of $\mathcal{L}^{\mathcal{G},0}$ by setting
    \begin{align*}
            g_{\sigma}(x_i)&=x_i,  \ \  i=1,2,  \\
            g_{\sigma}(y)&=y, \\
            g_{\sigma}(z)&=z, \\
            g_{\sigma}(t)&=t, \\
            g_{\sigma}(a_{v})&=a_{\sigma(v)}, \ \ v\in V(\mathcal{G}), \\
            g_{\sigma}(b)&=b, \\
            g_{\sigma}(b_{(v,w)})&=b_{(\sigma(v),\sigma(w))}, \ \  (v,w)\in E(\mathcal{G}).
    \end{align*}
    It is straightforward to verify that $g_{\sigma}$ is an automorphism of $\mathcal{L}^{\mathcal{G},0}$.

    Note that $g_{\sigma}(d\phi_{1,\alpha})$ is a cycle in $\mathcal{L}^{\mathcal{G},0}$ for $\alpha \in I_1$. According to the differential of $\mathcal{L}^{\mathcal{G},1}$, there exists a unique $\psi_{1,\alpha}\in \Phi_{|b|+1}$ such that $g_{\sigma}(d\phi_{1,\alpha})=d\psi_{1,\alpha}$. Thus we can define $g_{\sigma,1}$ on $\mathcal{L}^{\mathcal{G},1}$ as the extension of $g_{\sigma}$ by setting
    $$g_{\sigma,1}\big|_{\mathcal{L}^{\mathcal{G},0}}=g_{\sigma},\ \ g_{\sigma,1}(\phi_{1,\alpha})=\psi_{1,\alpha} \ \ \mathrm{for}\  \alpha\in I_1.$$
    Since $g_{\sigma}$ is an automorphism of $\mathcal{L}^{\mathcal{G},0}$, it induces an isomorphism on the free $R$-module $Z_{|b|}(\mathcal{L}^{\mathcal{G},0})$. Thus, we conclude that $\{\psi_{1,\alpha} \}_{\alpha\in I_1}$ also forms a free basis for $\Phi_{|b|+1}$, which implies $g_{\sigma,1}$ is also an automorphism.

    By similar arguments, we can extend $g_{\sigma,s}$ to an automorphism $g_{\sigma,s+1}$ on $\mathcal{L}^{\mathcal{G},s+1}$ inductively for $s=1,2,\dots,3n-1$. Finally, take $f_{\sigma}=g_{\sigma,3n}$ as required.

\end{proof}

For convenience, let $\mathcal{L}$ denote the Lie subalgebra of $\mathcal{L}^{\mathcal{G},0}$ generated by
$$\left\{x_1,x_2,y,z,t,a_{v}\ \middle|\  v\in V(\mathcal{G})\right\},$$
and let $\mathcal{T}$ denote the universal enveloping algebra of $\mathcal{L}$. Thus, $\mathcal{T}$ is a tensor algebra and can be considered as a Lie algebra via graded commutators
\begin{equation*}
[X,Y]=X\otimes Y-(-1)^{|X||Y|}Y\otimes X,\ \ \ X,Y\in \mathcal{T}.
\end{equation*}
We note that $\mathcal{L}$ is a Lie subalgebra of $\mathcal{T}$, and the only non-trivial differential on the generators of $\mathcal{L}$ is $dt=2x_2^2$. The inequality $|x_1|+|b|-1>|b|+3n$ implies that all decomposable elements of degree $\leq |b|+3n$ belong to $\mathcal{L}$. In the rest of the proof, by expanding Lie brackets as graded commutators, we can view such decomposable elements as elements in the tensor algebra $\mathcal{T}$.

Define the sets of Lie monomials
\begin{align*}
    S_1&=\left\{ (\ad z)^{5n-7}(y)   \right\}, \\
    S_2&=\left\{ (\ad z)^{n-6}(\ad x_{i_1})\dots (\ad x_{i_{4n+1}}) (y) \ \middle| \ i_1,\dots,i_{4n+1} =1 \ \mathrm{or } \ 2 \right\}, \\
    S_3&=\left\{ (\ad y)^{2n}(\ad z)^{n-5}(\ad x_1)^{2n-1}(x_2) \right\}, \\
    S_4&=\left\{ (\ad y)^{2n+1}(\ad z)^{n-6}(\ad x_1)^{2n-3}([x_2,t]) \right\}, \\
    S_5&=\left\{ (\ad a_v)(\ad a_w)(\ad a_u)(\ad x_1)^{2n-1} (\ad y)^{\Delta_1}(\ad z)^{\Delta_2}(x_2) \middle| v,w,u\in V(\mathcal{G})\right\}, \\
    S_6&=\left\{ (\ad a_{v})(\ad a_{w})(\ad x_1)^{2n-3}(\ad y)^{\Delta_3}(\ad z)^{\Delta_4}([x_2,t]) \ \middle| \ v,w\in V(\mathcal{G})\right\}.
\end{align*}
Here for convenience, we write
\begin{equation*}
    \begin{split}
        \Delta_1=& 2n-(3k-h_v-h_w-h_u), \\
        \Delta_2=& n-5-(h_v+h_w+h_u), \\
        \Delta_3=& 2n+1-(2k-h_v-h_w), \\
        \Delta_4=& n-6-(h_v+h_w).
    \end{split}
\end{equation*}

\begin{lemma}\label{lm3}
    With notation as above. The Lie monomials in
    $$S_1 \cup S_2 \cup S_3 \cup S_4 \cup S_5 \cup S_6$$
    are linearly independent.
\end{lemma}
\begin{proof}
    We first introduce a multigrading in $\mathcal{L}$ defined by the wordlengths of the generators $\{x_1,x_2,y,z,t,a_v\}$. More precisely, the multidegrees of the generators are given in the following table.

$$\begin{tabular}{c|c}
generators&multidegree\\
\hline
$x_1$&$(1,0,0,0,0,0)$\\
\hline
$x_2$&$(0,1,0,0,0,0)$\\
\hline
$y$&$(0,0,1,0,0,0)$\\
\hline
$z$&$(0,0,0,1,0,0)$\\
\hline
$t$&$(0,0,0,0,1,0)$ \\
\hline
$a_v,\ v\in V(\mathcal{G})$&$(0,0,0,0,0,1).$
\end{tabular}
$$

Therefore, the multidegrees of Lie monomials in each set are listed as follows.

$$\begin{tabular}{c|c}
sets&multidegree\\
\hline
$S_1$&$(0,0,1,5n-7,0,0)$\\
\hline
$S_2$&$(*,*,1,n-6,0,0)$\\
\hline
$S_3$&$(2n-1,1,2n,n-5,0,0)$\\
\hline
$S_4$&$(2n-3,1,2n+1,n-6,1,0)$\\
\hline
$S_5$&$(2n-1,1,*,*,0,3)$\\
\hline
$S_6$&$(2n-3,1,*,*,1,2)$
\end{tabular}
$$

Note that the multidegrees of $S_1,S_2,\dots,S_6$ are pairwise distinct. Thus, it suffices to show the linear independence for each set separately.

As an element in the tensor algebra $\mathcal{T}$, $(\ad z)^{5n-7}(y)$ in $S_1$ can be expanded as
\begin{equation*}
    \begin{split}
        (\ad z)^{5n-7}(y)&=\Big[z,\big[z,\dots[z,y]\dots\big]\Big] \\
        &=z^{5n-7}y+\mathrm{other \ summands}.
    \end{split}
\end{equation*}
In the expansion above, $z^{5n-7}y$ is the only summand with $y$ on the rightmost side, which implies that the element in $S_1$ is linearly independent.

For the element $(\ad z)^{n-6}(\ad x_{i_1})\dots (\ad x_{i_{4n+1}})(y)$ in $S_2$ with $i_1,i_2,\dots,i_{4n+1}\in\{1,2\}$, it can be expanded as
\begin{equation*}
    \begin{split}
        (\ad z)^{n-6}(\ad x_{i_1})\dots (\ad x_{i_{4n+1}})(y)&=\Big[z,\dots\big[x_{i_1},\dots[x_{i_{4n+1}},y]\dots\big]\dots \Big] \\
        &=z^{n-6}x_{i_1}x_{i_2}\dots x_{i_{4n+1}}y+\mathrm{other \ summands},
    \end{split}
\end{equation*}
and $z^{n-6}x_{i_1}x_{i_2}\dots x_{i_{4n+1}}y$ is the only summand with $y$ {on the rightmost side}. Since the elements
$$z^{n-6}x_{i_1}x_{i_2}\dots x_{i_{4n+1}}y,\ \ i_1,i_2,\dots,i_{4n+1}\in\{1,2\}$$
are linearly independent in $\mathcal{T}$, we conclude that the elements in $S_2$ are linearly independent.

By similar arguments, we can show the linear independence for $S_3,\dots,S_6$ separately, which completes the proof of Lemma \ref{lm3}.

\end{proof}

Recall that homotopy equivalences between $r$-mild dgl's are precisely quasi-isomorphisms \cite[p. 426]{AN1}. In particular, the automorphism $f_{\sigma}$ defined in Lemma \ref{hx1} is a self-homotopy equivalence of $\mathcal{L}^{\mathcal{G}}$. Then based on Lemma \ref{lm3}, we can prove the following lemmas.

\begin{lemma}\label{lm4}
    With notation as in Lemma \ref{hx1}. For every morphism $f$ with $[f]\in \mathcal{E}(\mathcal{L}^{\mathcal{G}})$, there exists $\sigma\in \mathrm{aut}(\mathcal{G})$ such that
    \begin{alignat*}{2}
        f(x_i)&=f_{\sigma}(x_i),\ \ \ \ && i=1,2, \\
        f(y)&=f_{\sigma}(y), \ \ \ \ && \\
        f(z)&=f_{\sigma}(z), \ \ \ \ && \\
        f(t)&=f_{\sigma}(t), \ \ \ \ && \\
        f(a_v)&=f_{\sigma}(a_v), \ \ \ \ && v\in V(\mathcal{G}), \\
        f(b)&=f_{\sigma}(b)+B, \ \ \ \ &&\\
        f(b_{(v,w)})&=f_{\sigma}(b_{(v,w)})+B_{(v,w)}, \ \ \ \ && (v,w)\in E(\mathcal{G}),
    \end{alignat*}
    where $B$ and $B_{(v,w)}$'s are boundaries.
\end{lemma}
\begin{proof}
By Lemma \ref{lm2}\,(2), we can write
\begin{align*}
        f(x_i)&=\sum_{j=1}^{2} \lambda_{ij}x_j,\ \  i=1,2, \\
        f(y)&=\mu y, \\
        f(z)&=\gamma z, \\
        f(t)&=\xi t, \\
        f(a_v)&=\sum_{v^\prime \in V_{h_v}}\delta_{vv^\prime} a_{v^\prime}+A_v, \ \  v\in V(\mathcal{G}), \\
        f(b)&=\zeta b+\sum_{(v^\prime,w^\prime)\in E(\mathcal{G})}\zeta_{(v^\prime,w^\prime)}  b_{(v^\prime,w^\prime)}+B, \\
        f(b_{(v,w)})&=\eta_{(v,w)} b+\sum_{(v^\prime,w^\prime)\in E(\mathcal{G})} \eta_{(v,w)(v^\prime,w^\prime)}b_{(v^\prime,w^\prime)}+B_{(v,w)}, \ \  (v,w)\in E(\mathcal{G}),
\end{align*}
where all coefficients belong to $R$, all sums have only finitely many nonzero summands and $A_v$'s, $B$, $B_{(v,w)}$'s are decomposable elements. Since $f$ is a homotopy equivalence between mild dgl's, it is a quasi-isomorphism as well. It follows that
$$\lambda_{11}\lambda_{22}-\lambda_{12}\lambda_{21}\neq 0, \ \ \mu\neq 0, \ \ \gamma\neq 0. $$
Since $df(t)=f(dt)$, we obtain
$$2\xi x_2^2=2(\lambda_{21}x_1+\lambda_{22}x_2)^2 $$
in the tensor algebra $\mathcal{T}$. Since $x_1^2$, $x_1x_2$, $x_2x_1$ and $x_2^2$ are linearly independent, we conclude that
$$\lambda_{21}=0, \ \  \xi=\lambda_{22}^2. $$
Note that $\lambda_{21}=0$ implies $\lambda_{11}\lambda_{22}\neq 0$.

As $df(b)=f(db)$, the two expressions below must be equal:
\begin{equation}\label{w1}
    \begin{split}
        df(b)&=\zeta \Big((\ad z)^{5n-7}(y)+(\ad z)^{n-6}(\ad x_1)^{4n+1}(y)+(\ad z)^{n-6}(\ad x_1)^{4n}([x_2,y]) \\
        &+(\ad y)^{2n}(\ad z)^{n-5}(\ad x_1)^{2n-1} (x_2) +(\ad y)^{2n+1}(\ad z)^{n-6}(\ad x_1)^{2n-3}([x_2,t])\Big) \\
        &+\sum_{(v^\prime,w^\prime)\in E(\mathcal{G})}\zeta_{(v^\prime,w^\prime)} \Big( (\ad a_{v^\prime})^3(\ad x_1)^{2n-1} (\ad y)^{2n-3(k-h_{v^\prime})}(\ad z)^{n-5-3h_{v^\prime}}(x_2)   \\
        &+ (\ad a_{v^\prime})(\ad a_{w^\prime})(\ad x_1)^{2n-3}(\ad y)^{2n+1-(2k-h_{v^\prime}-h_{w^\prime})}(\ad z)^{n-6-(h_{v^\prime}+h_{w^\prime})}([x_2,t]) \Big) \\
        &+dB,
    \end{split}
\end{equation}
\begin{equation}\label{w2}
    \begin{split}
        f(db)&=\mu \gamma^{5n-7} (\ad z)^{5n-7}(y)+\mu \gamma^{n-6} (\ad z)^{n-6}(\ad (\lambda_{11}x_1+\lambda_{12}x_2))^{4n+1}(y) \\
        &+\lambda_{22}\mu \gamma^{n-6} (\ad z)^{n-6}(\ad (\lambda_{11}x_1 + \lambda_{12}x_2))^{4n}([x_2,y]) \\
        &+\lambda_{22}\mu^{2n}\gamma^{n-5}  (\ad y)^{2n}(\ad z)^{n-5}(\ad (\lambda_{11}x_1 + \lambda_{12}x_2))^{2n-1} (x_2) \\
        &+\lambda_{22}^3\mu^{2n+1} \gamma^{n-6} (\ad y)^{2n+1}(\ad z)^{n-6}(\ad (\lambda_{11}x_1 + \lambda_{12}x_2))^{2n-3}([x_2,t]). \ \ \ \
    \end{split}
\end{equation}

\smallskip

Since $B$ is decomposable, we have $B\in \mathcal{L}\subset \mathcal{T}$, which implies that each summand appearing in $dB$ must contain a factor $x_2^2$. Thus, the Lie monomials in the expressions (\ref{w1}) and (\ref{w2}) with multidegree $(4n,1,1,n-6,0,0)$ are in $S_2$, which are linearly independent by Lemma \ref{lm3}. Note that the expansion of the term
$$\mu \gamma^{n-6} (\ad z)^{n-6}(\ad (\lambda_{11}x_1+\lambda_{12}x_2))^{4n+1}(y)$$
in (\ref{w2}) has the summand
$$\lambda_{11}^{4n}\lambda_{12} \mu \gamma^{n-6} (\ad z)^{n-6} (\ad x_2) (\ad x_1)^{4n}(y),$$
while no other expansions of the terms in (\ref{w1}) and (\ref{w2}) provide such a summand. It follows that its coefficient $\lambda_{11}^{4n}\lambda_{12} \mu \gamma^{n-6} $ must equal zero. From
$$\lambda_{11}\lambda_{22}\neq 0,\ \ \mu\neq 0, \ \ \gamma\neq 0,$$
we have $\lambda_{12}=0$. Recall that each summand appearing in $dB$ must contain a factor $x_2^2$. After substituting $\lambda_{12}=0$ to the expressions (\ref{w1}) and (\ref{w2}), we derive that all the Lie monomials in (\ref{w1}) and (\ref{w2}) are in $\cup_{i=1}^6S_i$, which are linearly independent by Lemma \ref{lm3}. Then by comparing the coefficients of the Lie monomials, we conclude that
\begin{equation}\label{e1}
    \begin{split}
    \zeta&=\mu \gamma^{5n-7}=\lambda_{11}^{4n+1}\mu\gamma^{n-6}=\lambda_{11}^{4n}\lambda_{22}\mu\gamma^{n-6} \\
    &= \lambda_{11}^{2n-1}\lambda_{22}\mu^{2n}\gamma^{n-5}=\lambda_{11}^{2n-3} \lambda_{22}^3\mu^{2n+1}\gamma^{n-6}
    \end{split}
\end{equation}
and $\zeta_{(v^\prime,w^\prime)}=0$ for $(v^\prime,w^\prime)\in E(\mathcal{G})$. From
$$\lambda_{11}^{4n+1}\mu\gamma^{n-6}=\lambda_{11}^{4n} \lambda_{22}\mu\gamma^{n-6},\ \ \lambda_{11}^{2n-1}\lambda_{22}\mu^{2n}\gamma^{n-5}=\lambda_{11}^{2n-3} \lambda_{22}^3\mu^{2n+1}\gamma^{n-6},$$
we have $\lambda_{11}=\lambda_{22}$, $\mu=\gamma$. Thus, (\ref{e1}) becomes
$$\zeta=\mu^{5n-6}=\lambda_{11}^{4n+1}\mu^{n-5}=\lambda_{11}^{2n}\mu^{3n-5}.$$
Therefore, we derive that
$$\mu^{5n-6}=\lambda_{11}^{4n+1}\mu^{n-5}\implies \mu^{4n-1}=\lambda_{11}^{4n+1},$$
$$\lambda_{11}^{4n+1}\mu^{n-5}=\lambda_{11}^{2n}\mu^{3n-5}\implies \lambda_{11}^{2n+1}=\mu^{2n}.$$
Thus, we can obtain $\lambda_{11}=\mu$, and 
$$\lambda_{11}^{2n+1}=\mu^{2n} \implies \lambda_{11}=1.$$
It follows that $\zeta=\lambda_{11}=\lambda_{22}=\mu=\gamma=1$, which implies $dB=0$. By Lemma \ref{lm2}\,(4), $B$ is a boundary.

For $df(b_{(v,w)})=f(db_{(v,w)})$, the following two expressions are equal:
\begin{equation}\label{ab1}
    \begin{split}
        d&f(b_{(v,w)})=\eta_{(v,w)} \Big((\ad z)^{5n-7}(y)+(\ad z)^{n-6}(\ad x_1)^{4n+1}(y) +(\ad z)^{n-6}(\ad x_1)^{4n}([x_2,y]) \\
        & +(\ad y)^{2n}(\ad z)^{n-5}(\ad x_1)^{2n-1} (x_2) +(\ad y)^{2n+1}(\ad z)^{n-6}(\ad x_1)^{2n-3}([x_2,t])\Big) \\
        &+\sum_{(v^\prime,w^\prime)\in E(\mathcal{G})} \eta_{(v,w)(v^\prime,w^\prime)}\Big( (\ad a_{v^\prime})^3(\ad x_1)^{2n-1} (\ad y)^{2n-3(k-h_{v^\prime})}(\ad z)^{n-5-3h_{v^\prime}}(x_2)  \\
        &+(\ad a_{v^\prime})(\ad a_{w^\prime})(\ad x_1)^{2n-3}(\ad y)^{2n+1-(2k-h_{v^\prime}-h_{w^\prime})}(\ad z)^{n-6-(h_{v^\prime}+h_{w^\prime})}([x_2,t]) \Big) \\
        &+ dB_{(v,w)} ,
    \end{split}
\end{equation}
\begin{equation}\label{ab2}
    \begin{split}
        f(db&_{(v,w)})=  \Big(\ad (\sum_{v^\prime \in V_{h_v}}\delta_{vv^\prime} a_{v^\prime}+A_v)\Big)^3(\ad x_1)^{2n-1} (\ad y)^{2n-3(k-h_v)}(\ad z)^{n-5-3h_v}(x_2) \\
        &+\Big(\ad (\sum_{v^\prime \in V_{h_v}}\delta_{vv^\prime} a_{v^\prime}+A_v)\Big)\Big(\ad (\sum_{w^\prime \in V_{h_w}}\delta_{ww^\prime} a_{w^\prime}+A_w)\Big) \\
        &\ \ \ \ \ \ \ \ \ \ \ \ \ \ \ \ \   \circ(\ad x_1)^{2n-3}(\ad y)^{2n+1-(2k-h_v-h_w)}(\ad z)^{n-6-(h_v+h_w)}([x_2,t]).
    \end{split}
\end{equation}

\smallskip

Similarly, since $B_{(v,w)}$ is decomposable, we have $B_{(v,w)}\in \mathcal{L}\subset \mathcal{T}$, which implies each summand in $dB_{(v,w)}$ must contain a factor $x_2^2$. Note that $A_v\in \mathbb{L}^{k-h_v,h_v}(y,z)$ by Lemma \ref{lm2}\,(3). Thus, all other terms in (\ref{ab1}) and (\ref{ab2}) have wordlength of $x_2$ at most one, which yields $dB_{(v,w)}=0$. Then by Lemma \ref{lm2}\,(4), each $B_{(v,w)}$ is a boundary. Then, only the Lie monomial $\eta_{(v,w)} (\ad z)^{5n-7}(y)$ has multidegree $(0,0,1,5n-7,0,0)$ in the expressions (\ref{ab1}) and (\ref{ab2}). It follows that $\eta_{(v,w)}=0$ by Lemma \ref{lm3}.

Recall that $\mathcal{G}$ is a connected graph with more than one vertex. For every vertex $v\in V(\mathcal{G})$, there exists an edge $(v,w)\in E(\mathcal{G})$ for some vertex $w\in V(\mathcal{G})$. Next, we will show that for any fixed vertex $v\in V(\mathcal{G})$ and fixed edge $(v,w)\in E(\mathcal{G})$, there exists a unique vertex $p\in V(\mathcal{G})$ and a unique edge $(p,q)\in E(\mathcal{G})$ satisfying:
\begin{equation}\label{xi}
    \begin{split}
        &A_v=0, \ \ \delta_{vp}=1, \ \ \delta_{vp^\prime}=0 \ \ \mathrm{for} \ p^\prime\neq p,   \\
        &\eta_{(v,w)(p,q)}=1, \ \ \eta_{(v,w)(p^\prime,q^\prime)}=0 \ \ \mathrm{for} \ (p^\prime,q^\prime)\neq (p,q).
    \end{split}
\end{equation}

Since $f$ induces an isomorphism on homology, there exists at least one nonzero coefficient among $\delta_{vv^\prime}$'s in the expression
$$f(a_v)=\sum_{v^\prime\in V_{h_v}}\delta_{vv^\prime}a_{v^\prime}+A_v, $$
denoted by $\delta_{vp}$. Then in the expression (\ref{ab2}), the expansion of the first summand
\begin{equation}\label{A}
    \Big(\ad (\sum_{v^\prime \in V_{h_v}}\delta_{vv^\prime} a_{v^\prime}+A_v)\Big)^3(\ad x_1)^{2n-1} (\ad y)^{2n-3(k-h_v)}(\ad z)^{n-5-3h_v}(x_2)
\end{equation}
provides a term
$$\delta_{vp}^2(\ad A_v)(\ad a_p)^2(\ad x_1)^{2n-1} (\ad y)^{2n-3(k-h_v)}(\ad z)^{n-5-3h_v}(x_2) $$
with multidegree $(2n-1,1,2n-2(k-h_v),n-5-2h_v,0,2)$, and there is no term in the expression (\ref{ab1}) with the same multidegree. By similar arguments as in Lemma \ref{lm3}, as an element in $\mathcal{T}$, the above term has the only summand
$$\delta_{vp}^2A_va_p^2 x_1^{2n-1} y^{2n-3(k-h_v)}z^{n-5-3h_v}x_2 $$
with $x_2$ on the rightmost side, which cannot be canceled by other terms in (\ref{ab2}). It follows that $A_v=0$. Similarly we can show that $A_w=0$. As a result, all the Lie monomials in the expressions (\ref{ab1}) and (\ref{ab2}) are in $S_5\cup S_6$, which are linearly independent by Lemma \ref{lm3}.

For the vertex $p^\prime \neq p$, the expansion of (\ref{A}) provides a term
$$\delta_{vp}^2\delta_{vp^\prime}(\ad a_p)^2(\ad a_{p^\prime})(\ad x_1)^{2n-1} (\ad y)^{2n-3(k-h_v)}(\ad z)^{n-5-3h_v}(x_2) ,$$
while no other expansions of the terms in (\ref{ab1}) and (\ref{ab2}) provide such a summand. Since $\delta_{vp}\neq 0$, we conclude that $\delta_{vp^\prime}=0$ for $p^\prime\neq p$.

Similarly, we denote by $\delta_{wq}$ the unique nonzero coefficient in the expression of $f(a_w)$. Then, (\ref{ab2}) becomes
\begin{equation}\label{ab3}
    \begin{split}
    f&(db_{(v,w)})=\delta_{vp}^3(\ad a_p)^3(\ad x_1)^{2n-1} (\ad y)^{2n-3(k-h_v)}(\ad z)^{n-5-3h_v}(x_2) \\
    &+\delta_{vp}\delta_{wq}(\ad a_p)(\ad a_q)(\ad x_1)^{2n-3}(\ad y)^{2n+1-(2k-h_v-h_w)}(\ad z)^{n-6-(h_v+h_w)}([x_2,t]).
    \end{split}
\end{equation}
By comparing the coefficients of the Lie monomials
$$(\ad a_{v^\prime})^3(\ad x_1)^{2n-1} (\ad y)^{2n-3(k-h_{v^\prime})}(\ad z)^{n-5-3h_{v^\prime}}(x_2) $$
and
$$(\ad a_{v^\prime})(\ad a_{w^\prime})(\ad x_1)^{2n-3}(\ad y)^{2n+1-(2k-h_{v^\prime}-h_{w^\prime})}(\ad z)^{n-6-(h_{v^\prime}+h_{w^\prime})}([x_2,t]) $$
for $(v^\prime,w^\prime)\in E(\mathcal{G})$ in (\ref{ab1}) and (\ref{ab3}), we conclude that
\begin{equation}\label{xi1}
    \eta_{(v,w)(p,q)}=\delta_{vp}^3=\delta_{vp}\delta_{wq}\neq 0,\ \ \eta_{(v,w)(p^\prime,q^\prime)}=0 \ \ \mathrm{for} \ (p^\prime,q^\prime)\neq (p,q).
\end{equation}
In particular, $\eta_{(v,w)(p,q)}\neq 0$ implies $(p,q)\in E(\mathcal{G})$. Since $\mathcal{G}$ is a graph (which implies that if $(v,w)\in E(\mathcal{G})$ then $(w,v)\in E(\mathcal{G})$ too), by similar arguments for the edge $(w,v)$ we deduce that
\begin{equation}\label{xi2}
    \delta_{wq}^3=\delta_{wq}\delta_{vp}.
\end{equation}
Combining (\ref{xi1}) and (\ref{xi2}) gives 
$$\eta_{(v,w)(p,q)}=\delta_{vp}=\delta_{wq}=1,$$
which proves (\ref{xi}).

By virtue of the uniqueness of $p$, we can define a morphism of graphs
\begin{equation*}
    \begin{split}
        \sigma : \mathcal{G} &\longrightarrow \mathcal{G}  \\
        v \  &\longmapsto \  p \\
        (v,w) &\longmapsto (p,q).
    \end{split}
\end{equation*}
    Our final step is to show that the above $\sigma$ is an automorphism of graphs. As $f$ is a homotopy equivalence, it induces an isomorphism on homology in degree $|a_v|$, which implies that $\sigma$ is a permutation of $V(\mathcal{G})$. As a result, $f$ is an isomorphism on $\mathcal{L}^{\mathcal{G}}_{<|b|}$. For any $(v,w)\in E(\mathcal{G})$, we claim that $(\sigma^{-1}(v),\sigma^{-1}(w))\in E(\mathcal{G})$. Indeed, since $db_{(v,w)}$ is a boundary of degree $|b|-1$, it follows that
    \begin{equation*}
    \begin{split}
    f&^{-1}(db_{(v,w)})=(\ad a_{\sigma^{-1}(v)})^3(\ad x_1)^{2n-1} (\ad y)^{2n-3(k-h_v)}(\ad z)^{n-5-3h_v}(x_2) \\
    &+(\ad a_{\sigma^{-1}(v)})(\ad a_{\sigma^{-1}(w)})(\ad x_1)^{2n-3} \\
    &\ \ \ \ \ \ \ \  \circ(\ad y)^{2n+1-(2k-h_v-h_w)}(\ad z)^{n-6-(h_v+h_w)}([x_2,t])
    \end{split}
    \end{equation*}
is also a boundary. However, the last term above only appears in the differential of $b_{(\sigma^{-1}(v),\sigma^{-1}(w))}$, and the linear independence by Lemma \ref{lm3} forces the presence of the generator $b_{(\sigma^{-1}(v),\sigma^{-1}(w))}$. Therefore, we can deduce that $(\sigma^{-1}(v),\sigma^{-1}(w))\in E(\mathcal{G})$. It follows from the definition that $\sigma\in \mathrm{aut}(\mathcal{G})$. This completes the proof of Lemma \ref{lm4}.

\end{proof}

\begin{lemma}\label{hx2}
    With notation as in Lemma \ref{hx1}. For every morphism $f$ with $[f]\in \mathcal{E}(\mathcal{L}^{\mathcal{G}})$, there exists $\sigma\in \mathrm{aut}(\mathcal{G})$ such that $f\simeq f_{\sigma}$.
\end{lemma}
\begin{proof}
In Lemma \ref{lm4}, we have concluded that
\begin{align*}
    f|_\mathcal{L}&=f_{\sigma}|_\mathcal{L}, \\
    f(b)&=f_{\sigma}(b)+B, \\
    f(b_{(v,w)})&=f_{\sigma}(b_{(v,w)})+B_{(v,w)}, \ \  (v,w)\in E(\mathcal{G}),
\end{align*}
where $B$ and $B_{(v,w)}$'s are boundaries in $ \mathcal{L}^{\mathcal{G}}$. By Lemma \ref{ho}, we have
$$f|_{\mathcal{L}^{\mathcal{G},0}}\simeq g_{\sigma}= f_{\sigma}|_{\mathcal{L}^{\mathcal{G},0}},$$
By Theorem \ref{HEP}, we deduce that $f$ is homotopic to a morphism, denoted by $f_0$, whose restriction on $\mathcal{L}^{\mathcal{G},0}$ is $g_{\sigma}$.

We next compare the restriction of $f_0$ on $\mathcal{L}^{\mathcal{G},1}$ and $g_{\sigma,1}=f_{\sigma}|_{\mathcal{L}^{\mathcal{G},1}}$. Recall that $\{\phi_{1,\alpha}\}_{\alpha\in I_1}$ is a free basis for $\Phi_{|b|+1}$. For $\alpha\in I_1$, we have
$$d\left(f_0(\phi_{1,\alpha})-g_{\sigma,1}(\phi_{1,\alpha})\right)=f_0(d\phi_{1,\alpha})-g_{\sigma,1}(d\phi_{1,\alpha})=0,$$
since $d\phi_{1,\alpha}\in \mathcal{L}^{\mathcal{G},0} $, and $f_0,g_{\sigma,1}$ coincide on $\mathcal{L}^{\mathcal{G},0}$. By Lemma \ref{lm2}\,(4), it follows that $f_0(\phi_{1,\alpha})-g_{\sigma,1}(\phi_{1,\alpha})$ is a boundary. Applying Lemma \ref{ho} and Theorem \ref{HEP} again yields that
$$f_0|_{\mathcal{L}^{\mathcal{G},1}}\simeq g_{\sigma,1} =f_{\sigma} |_{\mathcal{L}^{\mathcal{G},1}},$$
and $f_0$ is homotopic to a morphism, denoted by $f_1$, whose restriction on $\mathcal{L}^{\mathcal{G},1}$ is $g_{\sigma,1}$.

Iterating the foregoing arguments, we obtain morphisms $f_1,f_2,\dots,f_{3n}$ such that
$$f_s|_{\mathcal{L}^{\mathcal{G},s}}=g_{\sigma,s}$$
for $s=1,2,\dots,3n$, and
$$f\simeq f_0\simeq f_1\simeq \dots \simeq f_{3n}. $$
Note that $f_{3n}$ and $f_{\sigma}$ coincide on $\mathcal{L}^{\mathcal{G},3n}=\mathcal{L}^{\mathcal{G}}$. Thus, we finally conclude that $f_{3n}= f_{\sigma}$, and hence $f\simeq f_{\sigma}$.

\end{proof}

\begin{lemma}\label{lm5}
    With notation as in Lemma \ref{hx1}. Define a map by setting
    \begin{equation*}
        \begin{split}
            \theta_{\mathcal{G}}: \mathrm{aut}(\mathcal{G})&\longrightarrow \mathcal{E}(\mathcal{L}^{\mathcal{G}}) \\
            \sigma\ \ &\longmapsto \ \  [f_{\sigma}].
        \end{split}
    \end{equation*}
    Then $\theta_{\mathcal{G}}$ is an isomorphism of groups.
\end{lemma}
\begin{proof}
    It is easy to verify that $\theta_{\mathcal{G}}$ is a group morphism, and Lemma \ref{hx2} implies $\theta_{\mathcal{G}}$ is surjective.

    To show that $\theta_{\mathcal{G}}$ is injective, suppose that $[f_{\sigma}]=[f_{\sigma^\prime}] $. Thus, we obtain  $H(f_{\sigma})=H(f_{\sigma^\prime})$. For $i=0,1,\dots, k$, we have
    $$H_{(k-i)|y|+i|z|}(\mathcal{L}^{\mathcal{G}})=\langle a_{v}  |\   v\in V_i \  \rangle \oplus \mathbb{L}^{k-i,i}(y,z).$$
    It follows that
    $$a_{\sigma(v)}=H(f_{\sigma})(a_v)=H(f_{\sigma^\prime})(a_v) =a_{\sigma^\prime(v)} $$
    for $v\in V_i$ and $i=0,1,\dots,k$. Since $V(\mathcal{G})= \cup_{i=0}^kV_i $, we conclude that $\sigma(v)=\sigma^\prime(v)$ for $v\in V(\mathcal{G})$, which implies $\sigma=\sigma^\prime$. This completes the proof of Lemma \ref{lm5}.

\end{proof}

Finally, we are in a position to prove the main results.

\begin{proof}[Proof of Theorem \ref{fi}]
    Set $r=8n^2-6n+1$ and $m=|b|+3n$. Recall that
    $$\rho(R) > 20n^3-33n^2+8n+4, $$
    and $\mathcal{L}^{\mathcal{G}}$ is $r$-mild by Lemma \ref{lm2}\,(1). By direct computation, we have
    $$\min\{ r+2 \rho(R)-3,r\rho(R)-1\}=r+2 \rho(R)-3>40n^3-58n^2+10n+6=m. $$
    According to \cite[Theorem 1]{AN2}, there exists an $r$-connected finite-dimensional CW-complex $X$, which satisfies $\mathcal{E}(X_{R})\cong \mathcal{E}(\mathcal{L}^{\mathcal{G}})\cong \mathrm{aut}(\mathcal{G}) \cong G$.

    Using the natural identification of the linearized chain complex $\mathcal{L}^{\mathcal{G}}/[\mathcal{L}^{\mathcal{G}},\mathcal{L}^{\mathcal{G}}]$ with the desuspension of the reduced cellular chain complex for $X_R$ \cite[Theorem 8.5]{AN1}, we can identify $\langle a_{v} | v\in V_i \rangle \cong \langle V_i \rangle \cong M_i$ with
    $$s^{-1}H_{(k-i)|y|+i|z|+1}(X_{R};\mathbb{Z})\cong H_{k(8n^2-2n-2)+i+1}(X_{R};\mathbb{Z}) $$
    for $i=1,2,\dots, k$. Their group actions are also equivalent due to the natural identification.

    In particular, there is a one-to-one correspondence between the generators of $\mathcal{L}^{\mathcal{G}}$ and the cells of $X$ except the $0$-cell. Consequently, when $G$ is a finite group and $\left\{M_i\right\}_{i=1}^{k}$ is a collection of finitely generated $RG$-modules, $\mathcal{G}$ can be chosen to be a finite graph, which implies $X$ is a finite CW-complex. This completes the proof of Theorem \ref{fi}.

\end{proof}

\begin{proof}[Proof of Corollary \ref{co1}]
    Set $k=2$. Choose an integer $n$ satisfying 
    $$n\geq11, \ \ 8n^2-6n+1\geq N, $$
    and put $P_{G,N}=20n^3-33n^2+8n+4$. For every subring $R\subset \Q$ with $\rho(R)>P_{G,N}$, we can apply Theorem \ref{fi} to
    $$(G,\{M_i\}_{i=1}^2)= (G, \{R,R\}),$$
    and obtain an $(8n^2-6n+1)$-connected finite-dimensional CW-complex $X$ such that $\mathcal{E}(X_{R})\cong G$. Since $8n^2-6n+1\geq N$, $X$ is also $N$-connected. When $G$ is finite, the two modules $R$ are finitely generated. Thus the final assertion of Theorem \ref{fi} gives a finite CW-complex $X$.

\end{proof}

\end{document}